\documentclass[a4paper,12pt]{article}
\usepackage[cp1251]{inputenc}
\usepackage[russian, ukrainian,english]{babel}
\usepackage{amsmath, amsthm, amsfonts, amssymb}
\usepackage{graphicx}

\theoremstyle{plain}
\newtheorem{theorem}{Theorem}[section]

\newtheorem{lemma}{Lemma}[section]

\newtheorem{remark}{Remark}[section]
\newtheorem{definition}{Definition}[section]

\begin{document}
	
	\begin{center}
		{\bf Stationary points in coalescing stochastic flows on $\mathbb{R}$}
		
		\vskip20pt
		
		A.~A.~Dorogovtsev, G.~V.~Riabov, B.~Schmalfu\ss
		
			\end{center}
		
		\vskip20pt

		{\small {\bf Abstract.} This work is devoted to long-time properties of the Arratia flow with drift -- a stochastic flow on $\mathbb{R}$ whose one-point motions are weak solutions to a stochastic differential equation $dX(t)=a(X(t))dt+dw(t)$ that move independently before the meeting time and coalesce at the meeting time. We study special modification of such flow (constructed in \cite{Riabov}) that gives rise to a random dynamical system and thus allows to discuss stationary points. Existence of a unique stationary point is proved in the case of a strictly monotone Lipschitz drift by developing a variant of a pullback procedure. Connections between the existence of a stationary point and properties of a dual flow are discussed.}

\section{Introduction}

In the present paper we investigate a long-time behaviour of the Arratia flow \cite{Arratia, Arratia2} and its generalizations -- Arratia flows with drifts. These objects will be introduced in the framework of stochastic flows on $\mathbb{R}$. Following \cite{LeJanRaimond}, by a stochastic flow on $\mathbb{R}$ we understand a family $\{\psi_{s,t}:-\infty<s\leq t<\infty\}$ of measurable random mappings of $\mathbb{R}$ that possess two properties. 

\begin{enumerate}
	
	\item Evolutionary property: for all $r\leq s\leq t,$ $x\in\mathbb{R},$  $\omega\in \Omega$ 
\begin{equation}
\label{eq23_03}
\psi_{s,t}(\omega,\psi_{r,s}(\omega,x))=\psi_{r,t}(\omega,x)
\end{equation}
and $\psi_{s,s}(\omega,x)=x.$

	\item Independent and stationary increments: for $t_1<\ldots<t_n$ mappings $\psi_{t_1,t_2},$ $ \ldots,$ $\psi_{t_{n-1},t_n}$ are independent and $\psi_{t_1,t_2}$ is equal in distribution to $\psi_{0,t_2-t_1}.$
	
\end{enumerate}

For each $n\geq 1$ an $\mathbb{R}^n$-valued  stochastic process $t\to (\psi_{s,t}(x_1),\ldots,\psi_{s,t}(x_n)),$ $t\geq s$ will be called an $n-$point motion of the stochastic flow  $\{\psi_{s,t}:-\infty<s\leq t<\infty\}$. In \cite[Th. 1.1]{LeJanRaimond} it is proved that distributions of all finite-point motions uniquely define the distribution of a stochastic flow, and, actually, one can construct a  stochastic flow by specifying distributions of all its finite-point motions in a consistent way. Using this approach we give a definition of the Arratia flow with drift. Throughout the paper $a:\mathbb{R}\to\mathbb{R}$ is a Lipschitz function. The Borel $\sigma$-field on $\mathbb{R}^n$ will be denoted by $\mathcal{B}(\mathbb{R}^n).$

Consider a SDE
\begin{equation}
\label{6_11_eq1}
dX(t)=a(X(t))dt+dw(t),
\end{equation}
where $w$ is a Wiener process. For every $x\in \mathbb{R}$ the equation \eqref{6_11_eq1} has a unique strong solution $\{X_x(t):t\geq 0\}$ and defines a Feller semigroup of transition probabilities on $\mathbb{R}$  \cite[Ch. V, Th. (24.1)]{RW2}
$$
P^{(1)}_t(x,A)=\mathbb{P}(X_x(t)\in A),
$$
$t\geq 0, x\in\mathbb{R}, A\in\mathcal{B}(\mathbb{R}).$

Further, 
$$
P^{(n),ind.}((x_1,\ldots,x_n),A_1
\times \ldots \times A_n)=\prod^n_{i=1}P^{(1)}_t(x_i,A_i),
$$
where $t\geq 0,$ $(x_1,\ldots,x_n)\in\mathbb{R}^n,$ $A_1,\ldots,A_n\in\mathcal{B}(\mathbb{R}^n),$ defines a Feller transition probability on $\mathbb{R}^n$ that corresponds to an $n$-dimensional SDE
$$
dX_i(t)=a(X_i(t))dt+dw_i(t), \ 1\leq i\leq n,
$$
where $w_1,\ldots,w_n$ are independent Wiener processes. The sequence $\{P^{(n),ind.}:n\geq 1\}$ is consistent \cite{LeJanRaimond}, i.e. given $t\geq 0,$ $x\in \mathbb{R}^n,$ $1\leq i_1< i_2<\ldots <i_k\leq n,$ $B_k\in \mathcal{B}(\mathbb{R}^k)$ and $C_n=\{y\in \mathbb{R}^n: (y_{i_1},\ldots,y_{i_k})\in B_k\},$  one has
$$
P^{(n),ind.}_t(x,C_n)=P^{(k),ind.}_t((x_{i_1},\ldots,x_{i_k}),B_k).
$$
Finite-point motions of the Arratia flow with drift $a$ are specified via the result of \cite{LeJanRaimond} (see also \cite[L. 4.1]{Riabov}).

\begin{lemma}\cite[Th. 4.1]{LeJanRaimond}
	There exists a unique consistent sequence of Feller transition semigroups $\{P^{(n),c}_t:n\geq 1\}$ (the so-called coalescing transition semigroups) such that
	
	\begin{enumerate}
		
		\item for every $n\geq 1$ $\{P^{(n),c}_t:t\geq 0\}$ is a transition semigroup on $\mathbb{R}^n;$
		
		\item for all $x\in \mathbb{R}$ and $t\geq 0$
		$$
		P^{(2),c}_t((x,x),\Delta)=1,
		$$
		where $\Delta=\{(y,y):y\in\mathbb{R}\}$ is a diagonal;

\item Given $x\in\mathbb{R}^n$ let  $X=(X_1,\ldots,X_n)$ be an $\mathbb{R}^n$-valued Feller process with the starting point $x$ and transition probabilities  $\{P^{(n),c}_t:t\geq 0\}$ and $\tilde{X}=(\tilde{X}_1,\ldots,\tilde{X}_n)$ be an  $\mathbb{R}^n$-valued Feller process with the starting point $x$ and transition probabilities  $\{P^{(n),ind.}_t:t\geq 0\}$.  Let 
$$
\tau=\inf\{t\geq 0|\exists i,j: \ 1\leq i<j\leq n, X_i(t)=X_j(t)\}
$$
be the first meeting time for processes $X_1,\ldots,X_n,$ and 
$$
\tilde{\tau}=\inf\{t\geq 0| \exists i,j: \  1\leq i<j\leq n,  \tilde{X}_i(t)=\tilde{X}_j(t)\}
$$
be the first meeting time for processes $\tilde{X}_1,\ldots,\tilde{X}_n.$ Then distributions of stopped processes $\{(X_1(t\wedge \tau),\ldots,X_n(t\wedge \tau)):t\geq 0\}$ and $\{(\tilde{X}_1(t\wedge \tilde{\tau}),\ldots,\tilde{X}_n(t\wedge \tilde{\tau})):t\geq 0\}$ coincide.

\end{enumerate}

\end{lemma}

\begin{definition}
	\label{def23_03} A stochastic flow  $\{\psi_{s,t}:-\infty<s\leq t<\infty\}$ is the Arratia flow with drift $a,$ if for all $s\in\mathbb{R},$ $n\geq 1$ and $x=(x_1,\ldots,x_n)\in\mathbb{R}^n$ the finite-point motion
	$$
	t\to (\psi_{s,s+t}(x_1),\ldots,\psi_{s,s+t}(x_n)), t\geq 0
	$$
	is a Feller process with a starting point $x$ and transition probabilities $\{P^{(n),c}_t:t\geq 0\}.$

\end{definition}

\begin{remark} Because of the coalescence random mappings $x\to\psi_{s,t}(x)$ are not bijections. Hence, the family $\{\psi_{s,t}:-\infty<s\leq t<\infty\}$ cannot be extended to all pairs $(s,t)\in \mathbb{R}^2$ in such a way that the evolutinary property \eqref{eq23_03} holds for all $r,s,t\in\mathbb{R}$. In the terminology of \cite{Kunita} the family  $\{\psi_{s,t}:-\infty<s\leq t<\infty\}$ should be called rather a semiflow. As we won't deal with bijective mappings in this paper, we will keep the term ``stochastic flow'' for the family $\psi$.
	\end{remark}

Informally, the Arratia flow with drift is a system of coalescing (weak) solutions of the stochastic differential equation \eqref{6_11_eq1} that start from every time-space point $(t,x)\in \mathbb{R}^2$ and move independently up to the moment of meeting. The case $a=0$ corresponds to the Arratia flow -- a system of coalescing Brownian motions that are independent before meeting time.
As it was mentioned, existence and uniqueness of the Arratia flow with drift follows from  general results  \cite[Th. 1.1, 4.1]{LeJanRaimond}. 

The main objective of the present work is to study the long-time behaviour of random dynamical systems generated by the Arratia flow with drift. Appropriate modifications of coalescing stochastic flows on $\mathbb{R}$ that are random dynamical systems (in the sense of L. Arnold) was constructed in \cite{Riabov}. For Arratia flows with drift the existence result can be stated as follows.

\begin{theorem}
	\label{thm1}
	\cite[Th. 1.1]{Riabov} There exists a probability space $(\Omega,\mathcal{F},\mathbb{P})$ equipped with a measurable group $(\theta_t)_{t\in\mathbb{R}}$ of measure-preserving transformations of $\Omega$ and a measurable mapping 
	$$
	\varphi:\mathbb{R}_+\times\Omega\times \mathbb{R}\to \mathbb{R}
	$$
	such that 
	\begin{enumerate}
		\item $\varphi$ is a perfect cocycle over $\theta:$
\begin{equation}
	\label{eq28_1}
		\forall s,t\geq 0, \omega\in \Omega, x\in \mathbb{R} \ \ \ \varphi(t+s,\omega,x)=\varphi(t,\theta_s \omega,\varphi(s,\omega,x));
\end{equation}
		
		\item a stochastic flow $\{\psi_{s,t}:-\infty<s\leq t<\infty\}$ defined by
		$$
		\psi_{s,t}(\omega,x)=\varphi(t-s,\theta_s \omega,x)
		$$
		is the Arratia flow with drift $a$ in the sense of definition \ref{def23_03}.
	\end{enumerate}
	
\end{theorem}

\begin{remark}\label{rem23_03} In the terminology of random dynamical systems, evolutionary property \eqref{eq23_03} of $\psi$ can be rephrased as a perfect cocycle property of $\varphi.$ In this paper we deal only with stochastic flows that satisfes the evolutionary property with no exceptions. Such modifications of coalescing stochastic flows already appeared in \cite{Arratia2} (for the Arratia flow) and in \cite{LeJanRaimond, Darling} (for general stochastic flows). Main distinction of a theorem \ref{thm1} modification is that it combines perfect cocycle property of $\varphi$ with the measurability of the group of shifts $\theta.$ It must be noted that a number of various modifications of the Arratia flow that do not deal with the group of shifts of underlying probability space appeared in \cite{Harris, FINR, NT, SSS, BGS}. 
	
\end{remark}

The long-time behaviour of the Arratia flow (in the driftless case $a=0$) looks simple. Indeed, trajectories in the Arratia flow move like independent Wiener processes before meeting, hence with probability $1$ each pair of trajectories meets in a finite time and coalesces into a one trajectory. From this point of view it is natural to ask whether there is a path in the Arratia flow which is infinite in both directions, i.e. is there a (random) continuous  function $t\to \eta_t(\omega),$ $t\in\mathbb{R},$ such that for all $t\in\mathbb{R},$ $\omega\in\Omega$  there exist (random) $s\leq t$ and $x\in\mathbb{R}$ with
$$
\eta_t(\omega)=\psi_{s,t}(\omega,x), \ t\geq s?
$$
If true this would indicate the existence of a stationary point for the corresponding random dynamical system in the sense of the following standard definition.

\begin{definition}
	\label{def1} \cite[\S 1.4]{Arnold} A random variable $\eta$ is a stationary point for the random dynamical system $\varphi,$ if there exists a forward-invariant set of full-measure $\Omega_0\in \mathcal{F}$  (i.e. $\theta_t(\Omega_0)\subset \Omega_0$ for all $t\geq 0$), such that for all $\omega\in\Omega_0$ and $t\geq 0$
	$$
	\varphi(t,\omega,\eta(\omega))=\eta(\theta_t \omega).
	$$
\end{definition}

Importance of stationary points for random dynamical systems  stems from the fact that they define invariant measures for the skew-product flow $\Theta_t(\omega,x)=(\theta_t \omega,\phi(t,\omega,x))$ by the relation $\mu(d\omega,dx)=\delta_{\eta(\omega)}(dx)\mathbb{P}(d\omega).$  Conversely, for a strictly monotone continuous random dynamical system on $\mathbb{R}$ every ergodic invariant measure for the skew-product flow $\Theta$ is generated by some stationary point \cite[Th. 1.8.4]{Arnold}.

The main question we address in this paper is the existence (and uniqueness) of a stationaty point for an Arratia flow with drift. Comparing to the well-studied case of continuous random dynamical systems \cite{Arnold,DimitroffScheutzow, CrauelDimitroffScheutzow} there are two new effects brought by coalescence. Consider the case of the Arratia flow $\{\psi_{s,t}:-\infty<s\leq t <\infty\}$ at first, i.e. assume that the drift $a=0$.  
From continuity of trajectories and the perfect cocycle property \eqref{eq28_1} it follows that mappings $x\to \psi_{s,t}(x)$ are monotone. Typical approach of proving existence of a stationary point in this case is to apply the so-called pullback procedure \cite{Khasminskii,ArnoldSchmalfuss}: if the  limit  $\lim_{t\to\infty}\psi_{-t,0}(x)$ is shown to exist (in a suitable sense) and to be independent from $x$, then the random variable $\eta=\lim_{t\to\infty}\psi_{-t,0}(x)$ is a candidate to be a stationary point. However, for the Arratia flow $\psi$ the pullback procedure is inapplicable. Now, $h\to \psi_{s,s+h}(x)-x,$ $h\geq 0,$ is a Wiener processes and  $\psi_{-t,0}(x)$ is a random variable with Gaussian distribution and variance $t.$ Therefore the limit in the pullback procedure doesn't exist even in the sense of weak convergence. In fact, as we show in theorem \ref{thm3} of section 3, the Arratia flow does not possess a stationary point. Our approach is based on the properties of the dual flow. It is well-known \cite{Arratia, Arratia2, SoucaliucTothWerner} that one can construct simultaneously two Arratia flows  $\{\psi_{s,t}:-\infty<s\leq t<\infty\}$ and $\{\tilde{\psi}_{s,t}:-\infty<s\leq t<\infty\}$ with non-crossing trajectories, i.e. for any two starting points $(r,x)$ and $(s,y)$ with $r<s$ there are no two values $t_1,t_2\in[r,s]$ such that 
$$
\psi_{r,t_1}(x)>\tilde{\psi}_{-s,-t_1}(y) \mbox{ and }
\psi_{r,t_2}(x)<\tilde{\psi}_{-s,-t_2}(y).
$$
More precise, the non-crossing property holds for the flow $\{\psi_{s,t}:-\infty<s\leq t<\infty\}$ and the flow with a reversed time $\{\tilde{\psi}_{-s,-t}:-\infty<s\leq t<\infty\}.$  As it is shown in the section 3, the coalescing property for the dual flow $\tilde{\psi}$ violates the possibility for a stationary point in the Arratia flow $\psi.$

In order to make a pullback procedure convergent we consider the Arratia flow with drift $a$ that satisfies the strict  monotonicity condition 
\begin{equation}
\label{7_11_eq2}
(a(x)-a(y))(x-y)\leq -\lambda (x-y)^2.
\end{equation}
In section 2 it is proved that under this condition the limit $\eta=\lim_{t\to \infty}\psi_{-t,0}(x)$ in the pullback procedure exists (and is independent from $x$). In the case of a continuous flow the stationarity of $\eta$ immediately follows from the construction. However, for the Arratia flow with drift every mapping $\psi_{s,t},$ $s<t,$ is a.s. a step function \cite{Riabov, Dorogovtsev}. As a consequence, existence of the limit in a pullback procedure does not directly imply that the limit is a stationary point. To overcome this difficulty  a detailed analysis of the convergence is needed. It is done in theorem \ref{thm2} of section 2 where we prove the main result of the paper.

{\bf Theorem.}
{\it Let $\varphi$ be a random dynamical system that corresponds to the Arratia flow with the drift $a$ (in the sense of Theorem \ref{thm1}). Assume that the drift $a$ is Lipschitz and for some $\lambda>0$ and all $x,y\in \mathbb{R}$ one has
$$
(a(x)-a(y))(x-y)\leq -\lambda (x-y)^2.
$$
Then there exists a unique stationary point $\eta$ for the random dynamical system $\varphi.$ }

The condition \eqref{7_11_eq2} is one of the easiest conditions used in the theory of continuous random dynamical systems  to prove the existence of random attractors, see \cite{FlandoliGessScheutzow} and references therein. But the discontinuity of mappings $x\to \varphi(t,\omega,x)$ makes it impossible to apply well-known results on long-time behaviour of order-preserving random dynamical systems \cite{ArnoldSchmalfuss, FlandoliGessScheutzow, ArnoldChueshov, Chueshov} in our situation. The question about a weaker sufficient condition, e.g. 
$$
\limsup_{|x|\to \infty} \frac{a(x)x}{|x|^{1+\kappa}}<0, \ \kappa>0,
$$
\cite{Khasminskii, FlandoliGessScheutzow, Veretennikov}, will be studied in our future work. Another open problem is an adaptation of methods different from the pullback procedure \cite{Schmalfuss1, Schmalfuss2} to the existence of stationary points in coalescing stochastic flows.

Results of sections 2 and 3  lead naturally to the question about the dual flow for the Arratia flow with drift $a.$ In section 4 we show that finite-point motions of such flow come from the Arratia flow with drift $-a.$ In particular, under condition \eqref{7_11_eq2} trajectories in the dual flow do not meet with positive probability contrary to the case of the Arratia flow.

\section{Stationary point for an Arratia flow with drift.}

In this section we prove existence and uniqueness of a stationary point for the Arratia flow with a strictly monotone drift. Assume that $a:\mathbb{R}\to\mathbb{R}$ is a Lipschitz function that satisfies the condition \eqref{7_11_eq2}, i.e. 
$$
(a(x)-a(y))(x-y)\leq -\lambda (x-y)^2
$$
for some $\lambda>0$ and all $x,y\in\mathbb{R}.$ In this section $\{\psi_{s,t}:-\infty<s\leq t<\infty\}$ will denote the Arratia flow with drift $a$ in the sense of definition \ref{def23_03}. We will assume that the flow is given by a random dynamical system $\varphi$ (see theorem \ref{thm1} of the introduction),
$$
\psi_{s,t}(\omega,x)=\varphi(t-s,\theta_s\omega,x).
$$
The stationary point $\eta$ will be constructed as a limit in a pullback procedure, i.e. we will prove that 
\begin{equation}
\label{eq28_2}
\eta=\lim_{t\to \infty}\psi_{-t,0}(x)
\end{equation}
exists a.s. As it was mentioned in the introduction, stationarity of $\eta$ does not follow directly from this construction. To prove it we refine convergence in \eqref{eq28_2}. It appears that rather strong stabilization takes place: with probability $1$ for all $c>0$ and all $t\geq t_0(\omega,c)$ one has 
$$
\psi_{-t,0}(\omega,[-c,c])=\{\eta(\omega)\}.
$$
Then one can indeed pass to the limit as $t\to\infty$ in the relation
$$
\psi_{-t,0}(\theta_h\omega,x)=\varphi(h,\omega,\psi_{-t+h}(\omega,x))
$$
and deduce that $\eta$ is a stationary point. In a sense, coalescence replaces continuity for the random dynamical system $\varphi$. 

To establish these results we estimate the first meeting time of two trajectories $\psi_{0,\cdot}(x)$ and $\psi_{0,\cdot}(y)$ by comparing their difference with an Ornstein-Uhlenbeck process (lemma \ref{lem2}). 
Now we are in a position to formulate and prove the main result.

\begin{theorem}
	\label{thm2} Let $\varphi$ be a random dynamical system that corresponds to the Arratia flow with the drift $a$ (in the sense of theorem \ref{thm1}). Assume that the drift $a$ is Lipschitz and for some $\lambda>0$ and all $x,y\in \mathbb{R}$ one has
	$$
	(a(x)-a(y))(x-y)\leq -\lambda (x-y)^2.
	$$
	Then there exists a unique stationary point $\eta$ for the random dynamical system $\varphi.$
\end{theorem}

\begin{proof} 
	
	Given a Wiener process $w$ we will denote by $X_x(t)$ a strong solution of the equation 
	\eqref{6_11_eq1} that starts from the point $x$ at a time $0,$ i.e.
	\begin{equation}
	\label{9_11_eq4}
	\begin{cases}
	dX_x(t)=a(X_x(t))dt+dw(t), \ t\geq 0 \\
	X_x(0)=x.
	\end{cases}
	\end{equation}

The following estimate is well-known, we refer to \cite[Ch. 4]{Mao} for the details.

\begin{lemma}
	\label{lem1}
	There exists $C>0$ such that for all $x\in \mathbb{R}$ and $t\geq 0$
	$$
	\mathbb{E}X^2_x(t)\leq C(1+x^2)
	$$
\end{lemma}

Next we estimate the distribution of the meeting time for processes from the flow. Let $X_x$ and $X_y$ be solutions to \eqref{9_11_eq4} with starting points $x$ and $y$ and independent Wiener processes $w_1$ and $w_2,$ respectively. Denote by $\tau_{x,y}$ the meeting time of $X_x$ and $X_y:$
$$
\tau_{x,y}=\inf\{t\geq 0: X_x(t)=X_y(t)\}.
$$
Let $g(t;x,y)=\mathbb{P}(\tau_{x,y}>t).$

\begin{lemma}
\label{lem2}
	There exists a constant $C>0$ such that for all $t\geq \frac{\log 2}{2\lambda}$ and all $x,y$ one has
	$$
	g(t;x,y)\leq C|x-y|e^{-\lambda t}
	$$
\end{lemma}

\begin{proof}
	Assume that $x>y.$ Consider a Wiener process $B(t)=\frac{w_1(t)-w_2(t)}{\sqrt{2}}$ and represent the  difference $X_x(t)-X_y(t)$ in the form
	$$
	\frac{X_x(t)-X_y(t)}{\sqrt{2}}=\frac{x-y}{\sqrt{2}}+\int^t_0\frac{a(X_x(s))-a(X_y(s))}{\sqrt{2}}ds+B(t)
	$$
Introduce the Ornstein-Uhlenbeck process $Z$ governed by the Wiener process $B,$  
	$$
	Z(t)=\frac{x-y}{\sqrt{2}}-\lambda \int^t_0 Z(s)ds+B(t),
	$$
	and consider the difference $D(t)=\frac{X_x(t)-X_y(t)}{\sqrt{2}}-Z(t).$ For all $t< \tau_{x,y}$ we have $X_x(t)>X_y(t)$ and the strict monotonicity condition \eqref{7_11_eq2} implies 
	$$
	D'(t)=\frac{a(X_x(t))-a(X_y(t))}{\sqrt{2}}+\lambda Z(t)\leq -\lambda D(t).
	$$
As $D(0)=0$ it follows that $D(t)\leq 0$ for $t\in[0,\tau_{x,y}].$ Hence, for $t\in [0,\tau_{x,y})$ $Z(t)> 0$ and the moment of meeting $\tau_{x,y}$ is less than the moment when $Z(t)$ hits zero. The distribution density $p(t)$ of the latter moment is well known \cite{PitmanYor}:
	$$
	p(t)=\frac{|x-y|}{2\sqrt{\pi}}\bigg(\frac{2\lambda}{e^{\lambda t}-e^{-\lambda t}}\bigg)^{3/2}\exp\bigg(-\frac{\lambda (x-y)^2 e^{-\lambda t}}{2(e^{\lambda t}-e^{-\lambda t})}+\frac{\lambda t}{2}\bigg)
	$$
	For $t\geq \frac{\log 2}{2\lambda}$ one has $e^{\lambda t}-e^{-\lambda t}\geq \frac{1}{2}e^{\lambda t}$ and 
	$$
	p(t)\leq C|x-y| e^{-\lambda t}
	$$
	Consequently,
	$$
	g(t;x,y)=\mathbb{P}(\tau_{x,y}>t)\leq \int^\infty_t p(s)ds\leq \frac{C}{\lambda}|x-y|e^{-\lambda t}.
	$$
	
\end{proof}

Applying lemma \ref{lem2} to  two-point motions of the Arratia flow with drift $\psi$ we deduce that 
$$
\mathbb{P}(\psi_{-s,0}(x)\ne \psi_{-s,0}(y))=g(s;x,y)\leq C|x-y|e^{-\lambda s}.
$$
In the next lemma we obtain similar estimate for trajectories that started at distinct times. It is done using independence and stationarity of increments of the flow and lemma \ref{lem1}.

\begin{lemma}
\label{lem3}
	There exists a constant $C>0$ such that for any $s\geq \frac{\log 2}{2\lambda}$ and all $t\geq s,$ $x,y\in\mathbb{R}$ one has
\begin{equation}
\label{eq28_4}
	\mathbb{P}(\psi_{-t,0}(x)\ne \psi_{-s,0}(y))\leq C(1+|x|+|y|)e^{-\lambda s}.
\end{equation}
\end{lemma}

\begin{proof} Evolutionary property \eqref{eq23_03} implies that 
	$$
	\mathbb{P}(\psi_{-t,0}(x)\ne \psi_{-s,0}(y))=	\mathbb{P}(\psi_{-s,0}(\psi_{-t,-s}(x))\ne \psi_{-s,0}(y))=
	\mathbb{E}g(s;\psi_{-t,-s}(x),y).
	$$
	Combining estimates of lemmata \ref{lem1}, \ref{lem2} we obtain
	$$
	\begin{aligned}
	\mathbb{P}(\psi_{-t,0}(x)\ne \psi_{-s,0}(y))&\leq Ce^{-\lambda s}\mathbb{E}|\psi_{-t,-s}(x)-y|=Ce^{-\lambda s}\mathbb{E}|X_x(t-s)-y| \\
	& \leq C_1 (1+|x|+|y|)e^{-\lambda s},
	\end{aligned}
	$$
	with some different constant $C_1.$

\end{proof}

From the lemma \ref{lem3} it follows that the limit $\lim_{n\to\infty}\psi_{-n,0}(n)$
exists with probability $1.$ Indeed, inequality \eqref{eq28_4} implies that 
$$
\sum^\infty_{n=1} \mathbb{P}(\psi_{-n-1,0}(n+1)\ne \psi_{-n,0}(n))<\infty.
$$
By the Borel-Cantelli lemma, with probability 1 there exists $N=N(\omega)$ such that  the sequence $\{\psi_{-n,0}(\omega,n):n\geq N(\omega)\}$ is constant. We denote its limit as $\eta,$
$$
\eta(\omega)=\lim_{n\to\infty}\psi_{-n,0}(\omega,n).
$$
Further, $\eta(\omega)=\lim_{n\to\infty}\psi_{-n,0}(\omega,-n)$ a.s. Indeed, by the lemma 2.3
	$$
	\mathbb{P}(\psi_{-n,0}(n)\ne\psi_{-n,0}(-n))\leq C(1+2n)Ce^{-\lambda n}
	$$
for large enough $n.$  

From obtained convergences  and monotonicity of trajectories it follows that with probability 1 for every $x\in\mathbb{R}$ $\psi_{-n,0}(\omega,x)=\eta(\omega),$ $n\geq N(\omega,x).$ 

Next we show that convergence to $\lim_{n\to\infty}\psi_{-n,0}(x)=\eta$ can be strengthened to convergence along real numbers $t\to\infty.$ In order to do it we prove that for every fixed $c$ with probability 1
	
$\exists n_0: \  \forall n\geq n_0 \ \forall t\in[0,1]$
	$$
	\psi_{-n,-n+t}(\omega,n)>c, \ \ \psi_{-n,-n+t}(\omega,-n)<c.
	$$
This result means that trajectories that start at large  (negative) moments of time from far positions can't reach fixed level $c$ in a bounded time. 

It is sufficient to consider only the first relation and large enough $c$. The assertion will follow from convergence of the series
	\begin{equation}
	\label{eq11_08}
	\sum^\infty_{n=1}\mathbb{P}(\exists t\in[0,1] \ \psi_{-n,-n+t}(n)\leq c)<\infty.
	\end{equation}
Condition \eqref{7_11_eq2} implies that  $a(x_0)=0$ for some $x_0\in\mathbb{R}.$ Assume that $c>x_0.$ Let $\Lambda$ be the Lipschitz constant for $a:$
$$
|a(x)-a(y)|\leq \Lambda|x-y|.
$$
For $n>\max(c,2(e^\Lambda(c-x_0)+x_0))$ we will estimate the probability
$$
	\mathbb{P}(\exists t\in[0,1] \ \psi_{-n,-n+t}(n)\leq c)=\mathbb{P}(\exists t\in[0,1] \ X_n(t)\leq c).
	$$
Recall that the process $\{X_n(t):t\geq 0\}$ is defined by the equation
	$$
	X_n(t)=n+\int^t_0 a(X_n(s))ds+w(t)
	$$
	with some Wiener process $w.$ Again, let us denote by $Z$ the Ornstein-Uhlenbeck process, governed by the Wiener process $w,$
	$$
	Z(t)=n-x_0-\Lambda\int^t_0Z(s)ds+w(t).
	$$
Consider the moment $\tau_c$ when the process $X_n$ hits the level $c.$ For all $t\leq \tau_c$ we have $X_n(t)\geq c>x_0$ and
	$$
	a(X_n(t))=a(X_n(t))-a(x_0)\geq -\Lambda (X_n(t)-x_0).
	$$
Then the derivative of the expression
	$$
	Z(t)-X_n(t)+x_0=\int^t_0(-\Lambda Z(s)-a(X_n(s)))ds
	$$
satisfies 
	$$
	(Z(t)-X_n(t)+x_0)'=-\Lambda Z(t)-a(X_n(t))\leq -\Lambda (Z(t)-X_n(t)+x_0).
	$$
	So, for all $t\leq \tau_c$ we have $Z(t)\leq X_n(t)-x_0.$ In particular,
	$$
	\mathbb{P}(\exists t\in[0,1] \ X_n(t)\leq c)\leq \mathbb{P}(\exists t\in[0,1] \ Z(t)\leq c-x_0)=
	$$
	$$
	=\mathbb{P}\bigg(\exists t\in[0,1] \ \ \ e^{-\Lambda t}(n-x_0)+\int^t_0 e^{-\Lambda(t-s)}dw(s)\leq c-x_0\bigg)\leq 
	$$
	$$
	\leq \mathbb{P}\bigg(\exists t\in[0,1] \ \ \ \int^t_0 e^{\Lambda s}dw(s)\leq e^\Lambda( c-x_0)+x_0-n\bigg)\leq 
	$$
	$$
	\leq \mathbb{P}\bigg(\exists t\in[0,1] \ \ \ \int^t_0 e^{\Lambda s}dw(s)\leq -\frac{n}{2}\bigg)
	$$
	Representing the integral $\int^t_0 e^{\Lambda s}dw(s)$ as a Wiener process with changed time and using tail estimates for minimum of a Wiener process \cite[Prop. 11.13]{K}, we obtain
	$$
	\mathbb{P}(\exists t\in[0,1] \ X_n(t)\leq c)\leq \mathbb{P}\bigg(\min_{0\leq t\leq \frac{e^{2\Lambda}-1}{2\Lambda}}w(t)\leq -\frac{n}{2}\bigg)\leq ce^{-a n^2}
	$$
	The sum in \eqref{eq11_08} is convergent.

Summarizing obtained results, following two properties hold on a set of probability 1.

\begin{itemize}
\item There exists  $N(\omega)$ such that for all $n\geq N(\omega)$
$$
\psi_{-n,0}(\omega,n)=\psi_{-n,0}(\omega,-n)=\eta(\omega).
$$

\item For every $c\in\mathbb{R}$ there exists $M(\omega,c)$ such that  for all $n\geq M(\omega,c)$
	$$
	\min_{t\in[0,1]}\psi_{-n,-n+t}(\omega,n)>c, \  \max_{t\in[0,1]} \psi_{-n,-n+t}(\omega,-n)<c.
	$$
	
\end{itemize}

Fix $\omega$ such that these two properties hold. Given $c>0$ there exists $L(\omega)$ such that for all $n\geq L(\omega)$
$$
\psi_{-n,0}(\omega,n)=\psi_{-n,0}(\omega,-n)=\eta(\omega), 
$$
$$
\min_{t\in[0,1]}\psi_{-n,-n+t}(\omega,n)>c, \max_{t\in[0,1]} \psi_{-n,-n+t}(\omega,-n)<-c.
$$
If $t\geq L(\omega),$ $t=n-s$ with  $0\leq s\leq 1,$ then
$$
c<\psi_{-n,-t}(\omega,n)
$$
and from evolutionary property and monotonicity of trajectories we have 
$$
\psi_{-t,0}(\omega,c)\leq \psi_{-n,0}(\omega,n)=\eta(\omega).
$$
Similarly,
$$
\psi_{-t,0}(\omega,-c)\geq \psi_{-n,0}(\omega,-n)=\eta(\omega).
$$
Consequently, with probability 1 for every $c>0$ there exists $t_0$ such that for every $t\geq t_0$ and every $x\in[-c,c]$
$$
\psi_{-t,0}(\omega,x)=\eta(\omega),
$$
(see figure 1) and the set 
$$
\Omega_0=\{\omega\in \Omega: \exists t_0>0 \ \  \forall c>0 
$$
$$
\exists t_1\geq t_0 \ \ \forall t\geq t_1 \ \ \forall x\in[-c,c] \ \ \ \psi_{-t,0}(\omega,x)=\psi_{-t_0,0}(\omega,0)\}
$$
has probability $1$ (its measurability follows by restricting values of all variables to rational ones).  Also, $\Omega_0$ is $\theta_t-$invariant for $t\geq 0$.

\begin{figure}[h]
	\begin{center}
		\includegraphics[width=10cm]{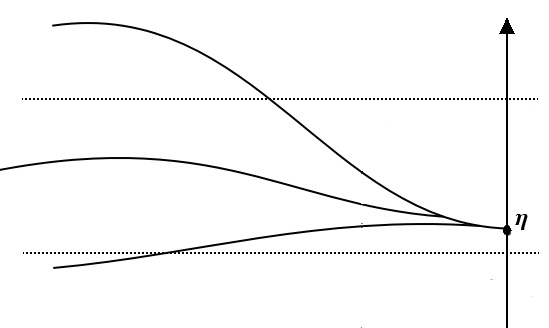}
		\caption{in sufficiently large time every set of trajectories started from a bounded set arrives to $\eta$}
	\end{center}
\end{figure}

Finally,  we show that the random variable $\eta$ is the needed stationary point. Recall that for any  $\omega\in\Omega_0$ for every $x\in\mathbb{R}$   there exists $t_0\geq 0$ such that for all $t\geq t_0$ one has
$$
\psi_{-t,0}(\omega,x)=\eta(\omega).
$$
Applying this property for $x=0,$ $\omega$ and $\theta_h \omega,$ we can find $t_0\geq 0$ such that for all $t\geq t_0$ one has 
$$
\psi_{-t,0}(\omega,0)=\eta(\omega), \psi_{-t,0}(\theta_h\omega,0)=\eta(\theta_h\omega).
$$ 
For such $t$ we have
$$
\varphi(h,\omega,\eta(\omega))=\psi_{0,h}(\omega,\eta(\omega))=\psi_{0,h}(\omega,\psi_{-t,0}(\omega,0))=
$$
$$
=\psi_{-t,h}(\omega,0)=\psi_{-t-h,0}(\theta_h\omega,0)=\eta(\theta_h\omega).
$$
In the last passage we used that $t+h\geq t_0.$ Existence of a stationary point is proved.

In order to show uniqueness, assume that $\eta_1$ and $\eta_2$ are stationary points for the random dynamical system $\varphi$. Given $\varepsilon>0$ find $c>0$ such that 
$$
\mathbb{P}(|\eta_1|>c)<\varepsilon, \mathbb{P}(|\eta_2|>c)<\varepsilon.
$$
Using the relation $\psi_{-t,0}(\omega,\eta_j(\theta_{-t}\omega))=\eta_j(\omega)$ and order-preserving property, we can estimate the difference $\eta_1-\eta_2$ as follows:
$$
\mathbb{P}(|\eta_1-\eta_2|>\delta)\leq 2\varepsilon+
$$
$$
+\mathbb{P}(\{\omega: |\psi_{-t,0}(\omega,\eta_1(\theta_{-t}\omega))-\psi_{-t,0}(\omega,\eta_1(\theta_{-t}\omega))|>\delta, |\eta_1(\theta_{-t}\omega)|\leq c, |\eta_2(\theta_{-t}\omega)|\leq c)\leq
$$
$$
\leq 2\varepsilon+\mathbb{P}(\psi_{-t,0}(c)-\psi_{_t,0}(-c)>\delta)
$$
The probability in the latter expression converges to zero, as $t\to \infty$ (lemma \ref{lem3}). It follows that $\eta_1=\eta_2$ a.s. The theorem is proved.

\end{proof}

\section{Non-existence of a stationary point for the Arratia flow}

 In this section we show that the Arratia flow does not possess  a stationary point. Throughout the section we assume that the drift in \eqref{6_11_eq1} is $a=0.$ As it was mentioned in the introduction, one can define simulatenously two Arratia flows $\{\psi_{s,t}:-\infty<s\leq t<\infty\}$ and  $\{\tilde{\psi}_{s,t}:-\infty<s\leq t<\infty\}$ in such a way that trajectories of $\psi$ (in forward time) and $\tilde{\psi}$ (in backward time) do not cross each other. Without loss of generality we may assume that the Arratia flow $\psi$ is the one given by the random dynamical system $\varphi,$ see the construction of $\varphi$  \cite[Th. 1.1]{Riabov}.  An expression for $\tilde{\psi}$ as a function of $\psi$ is given in  \cite{DK}. The joint distribution of forward and  backward trajectories was studied in \cite{SoucaliucTothWerner}. It was proved that in distribution finite-point motions of the Arratia flow and its dual  coincide with coalescing-reflecting Wiener processes, where the reflection is understood in the sense of Skorokhod \cite{Skorokhod}. 

\begin{theorem}
	\label{thm3} There is no stationary point in a random dynamical system $\varphi$ generated by the Arratia flow.
\end{theorem}

\begin{proof}  Assume on the contrary that a stationary point $\eta$ exists. Then $\eta$ is a random variable such that on the forward-invariant set $\Omega_0$ of full probability one has the following.
	$$
	\forall t\geq 0 \ \varphi(t,\omega,\eta(\omega))=\eta(\theta_t\omega), \ \omega\in\Omega_0.
	$$
At first we strengthen this property using that $(\theta_t)_{t\in\mathbb{R}}$ is a group of measure preserving transformations. Observe that the set $\Omega_1=\cap_{t\geq 0}\theta_t(\Omega_0)$ is $\theta_t$-invariant for all $t\in\mathbb{R}$ and has probability 1. For every $\omega\in\Omega_1$ and every $t\geq 0$ one has 
$$
\varphi(t,\theta_{-t}\omega,\eta(\theta_{-t}\omega))=\eta(\omega).
$$
In other words, with probability 1 for every $t\leq 0$ there exists $x\in \mathbb{R}$ such that 
$$
\psi_{-t,0}(x)=\eta.
$$

\begin{figure}[h]
\begin{center}
\includegraphics[width=10cm]{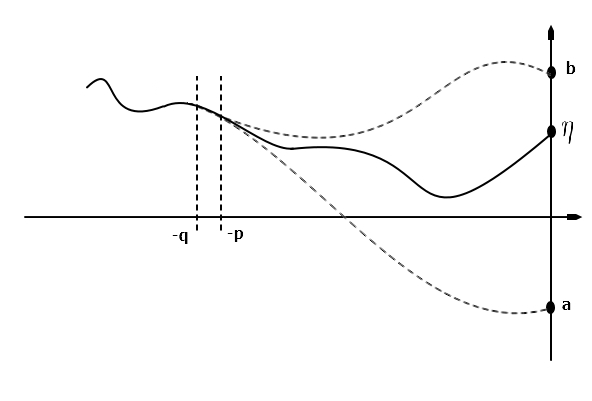}
\caption{dashed trajectories are from backward flow}
\end{center}
\end{figure}

Let rational points $a,b$ be such that $a<\eta<b$ with positive probability. Consider trajectories of the dual flow (in backward time) that start at time $0$ from the points $a$ and $b.$ As every two trajectories of the Arratia flow coalesce in a finite time, there is a time $p>0$ such that
$$
\tilde{\psi}_{0,p}(a)=\tilde{\psi}_{0,p}(b).
$$
For arbitrary $q>p$ consider the point $x$ such that $\psi_{-q,0}(x)=\eta.$ If $$
x>\tilde{\psi}_{0,q}(a)=\tilde{\psi}_{0,q}(b),
$$
then 
$$
\psi_{-q,0}(x)=\eta<b=\tilde{\psi}_{0,0}(b)
$$ and trajectories $t\to \psi_{-q,t}(x)$ and $t\to\tilde{\psi}_{0,-t}(b)$  intersect on the segment $[-q,0].$  Similarly, $x<\tilde{\psi}_{-q,0}(a)$ is impossible. So for all  $q>p$ we have $x=\tilde{\psi}_{0,q}(a)$ (see figure 2).  In particular, on the segment $[-q,-p]$ the trajectory $t\to \psi_{-q,t}(x)$ of the forward flow coincides with the trajectory $t\to \tilde{\psi}_{0,-t}(a)$ of the backward flow. It is impossible by \cite{SoucaliucTothWerner} and the theorem is proved.

	\end{proof}

\section{Dual flows}

Result of section 3 shows that at least for coalescing stochastic flows on $\mathbb{R}$ existence of a stationary point is connected with the structure of the dual flow. Taking into accout theorem \ref{thm1} of section 2 it is natural to ask what is the dual flow for the Arratia flow with drift $a.$  The next theorem describes finite-point motions of such dual flow. As above, we assume that $a:\mathbb{R}\to\mathbb{R}$   is a Lipschitz function.

\begin{theorem}
	For arbitrary $u_1,\ldots,u_n\in\mathbb{R},$ $s_1,\ldots,s_n\in \mathbb{R}$ there exist two families of random processes $\{f(s_i,t,u_i,): t\geq s_i, 1\leq i\leq n\}$ and $\{g(t,s_i,u_i): t\leq s_i, 1\leq i\leq n\}$ such that 
	\begin{enumerate}
		\item for every $i$ the process 
		$$
		f(s_i,t,u_i)-\int^t_{s_i}a(f(s_i,r,u_i))dr, \ t\geq s_i
		$$
		is a Wiener martingale with respect to the filtration
		$$
		\begin{aligned}
		\mathcal{F}^+_t=\sigma(&\{f(s_j,r,u_j):s_j\leq r\leq t,1\leq j\leq n\}\\
		&\cup \{g(r_1,s_k,u_k)-g(r_2,s_k,u_k):r_2\leq r_1\leq t, 1\leq k\leq n\});
		\end{aligned}
		$$
		
		\item for every $i$ the process 
		$$
		g(t,s_i,u_i)-\int^{s_i}_ta(g(r,s_i,u_i))dr, \ t\leq s_i
		$$
		is a Wiener martingale with respect to the filtration
		$$
		\begin{aligned}
		\mathcal{F}^-_t=\sigma(&\{g(r,s_j,u_j):t\leq r\leq s_j, 1\leq j\leq n\} \\
		&\cup \{f(s_k,r_2,u_k)-f(s_k,r_1,u_k):t\leq r_1\leq r_2, 1\leq k\leq n\});
		\end{aligned}
		$$
		
		\item $f(s_j,s_j,u_j)=g(s_j,s_j,u_j)=u_j,$ $j=1,\ldots,n;$
		
		\item for arbitrary $i_1\ne i_2$  processes $f(s_{i_1},\cdot,u_{i_1})$ and $f(s_{i_2},\cdot,u_{i_2})$
	coalesce after meeting, and processes  $g(\cdot,s_{i_1},u_{i_1})$ and $g(\cdot,s_{i_2},u_{i_2})$ coalesce after meeting;

		\item the trajectories of all processes $f$ and $g$ do not cross, i.e. there are no points $s_i\leq r_1<r_2\leq s_j$ such that 
		$$
		f(s_i,r_1,u_i)>g(r_1,s_j,u_j) \mbox{ and } f(s_i,r_2,u_i)<g(r_2,s_j,u_j)
		$$
		or 
		$$
		g(s_i,r_1,u_i)>f(r_1,s_j,u_j) \mbox{ and } g(s_i,r_2,u_i)<f(r_2,s_j,u_j);
		$$
		
		\item quadratic covariation of any two processes $f$ has a derivative $0$ before meeting time and $1$ after meeting time; quadratic covariation of any two processes $f$ has a derivative $0$ before meeting time and $1$ after meeting time.
		
				\end{enumerate}
\end{theorem}		
	\begin{proof}
		We present the construction of the process $f$ and $g$ in the following case. The processes $f$ will start from $u_1,\ldots,u_n$ at time $0$ and the processes $g$ will start (in backward time) from $v_1,\ldots,v_m$ at time $1.$ The general case can be obtained easily. To construct the desired set of processes we will use fractional step method proposed by P. Kotelenez for stochastic  differential equations with smooth coefficients \cite{GK} and successfully applied in \cite{DV} to the construction of the Arratia flow with Lipshitz drift $a.$
		
Let us take a partition $t_k=\frac{k}{n},$ $0\leq k\leq 2n.$ Denote by $\psi_0$ the Arratia flow and by $\tilde{\psi}_0$ the dual flow. Also denote by $h(s,t,u)$ the solution to Cauchy problem
		$$
		\begin{cases}
		dh(s,t,u)=a(h(s,t,u))dt, \\
		h(s,s,u)=u, \ t\geq s.
		\end{cases}
		$$
		
	Construct processes $\tilde{f}_n$ as a subsequent superposition of $\psi$ and $h$ on intervals of partition, e.g. 
	$$
	\tilde{f}_n(u)=\bigg(\psi_{t_{2n-1},t_{2n}}\circ h(t_{2n-2},t_{2n-1},\cdot)\circ \ldots \circ h(0,t_1,\cdot)\bigg)(u)
	$$
	Processes $\tilde{g}_n$ are defined in  the same way in backward time. Note that trajectories of $\tilde{f}_n$ and $\tilde{g}_n$ do not cross. Now define processes $f_n$ and $g_n$ by the rule 
	$$
	f_n(t)=\tilde{f}_n(2t), \ g_n(t)=\tilde{g}_n(2t), \ t\in[0,1].
	$$
	It follows from arguments in \cite{DV} that $\{f_n,g_n\}$ weakly converge to the family of processes with desired properties.
		\end{proof}
	
	Consequently, at least in the sense of finite-point motions, dual flow has the same structure as initial one, but with the drift $-a.$ Now, let us note that under condition \eqref{7_11_eq2} two processes
	$$
	z_i(t)=u_i-\int^t_0a(z_i(s))ds+w_i(s), \ i=1,2
	$$
	with independent Wiener processes $w_1,w_2$ do not meet with positive probability. Taking into account considerations of section 3 this explains why for the Arratia flow with a strictly monotone drift $a$ we have a possibility for existence of a stationary point.

\end{document}